\documentclass[11pt]{article}

\usepackage{amsmath,amsthm,amsfonts,amssymb}
\usepackage[T1]{fontenc}
\usepackage[utf8]{inputenc}
\usepackage[english]{babel}
\usepackage{mathtools}
\usepackage{geometry}
\usepackage{tikz}
\usetikzlibrary{calc,angles,quotes,patterns,decorations.pathreplacing,arrows.meta}
\usepackage[colorlinks=true,allcolors=blue]{hyperref}

\let\le\leqslant
\let\ge\geqslant

\theoremstyle{plain}
\newtheorem{theorem}{Theorem}[section]
\newtheorem{lemma}[theorem]{Lemma}
\newtheorem{proposition}[theorem]{Proposition}
\newtheorem{corollary}[theorem]{Corollary}
\theoremstyle{definition}
\newtheorem{definition}[theorem]{Definition}
\newtheorem{example}[theorem]{Example}
\theoremstyle{remark}
\newtheorem{remark}[theorem]{Remark}

\numberwithin{equation}{section}

\newcommand{\R}{\mathbb{R}}
\renewcommand{\P}{\mathbb{P}}
\providecommand{\Prob}{}\renewcommand{\Prob}[1]{\P\!\left\{#1\right\}}
\newcommand{\E}{\mathbb{E}}
\newcommand{\cN}{\mathcal{N}}
\newcommand{\rd}{\mathrm{d}}
\newcommand{\ind}{\mathbf{1}}
\newcommand{\eqd}{\stackrel{d}{=}}
\newcommand{\Leb}{\operatorname{Leb}}

\newcommand{\inv}{\operatorname{inv}}
\newcommand{\corr}{\operatorname{corr}}
\newcommand{\Arch}{\mathfrak{Arch}}
\newcommand{\Unif}{\mathit{Unif}}

\title{Geodesics in the continuous Hecke monoid}
\author{Sergei Presniakov\\ \small St.\ Petersburg State University \and Alexander Tiskin\\ \small St.\ Petersburg State University}
\date{}

\begin{document}
\maketitle

\begin{abstract}
A permuton is a probability measure on the unit square with uniform marginals, the limit shape
of a sequence of permutations. The sticky product of permutons is a continuous analogue of the
Demazure product in the $0$-Hecke monoid; on the distribution functions of permutons, that is,
on copulas, it is given by the tropical $(\min,+)$ product $\odot$. We show that every
absolutely continuous rotation-invariant probability measure $\rho$ on the plane, which we call
the \emph{profile}, generates a
one-parameter family of copulas $C^\rho=\{C^\rho_\beta\}_{\beta\in[0,\pi]}$. A copula in this
family is defined by a pair of projections of a random vector with distribution $\rho$ onto
directions at angle $\pi-\beta$. Along this family, the sticky product adds angles:
\[
  C^\rho_{\beta_1}\odot C^\rho_{\beta_2}=C^\rho_{\min(\beta_1+\beta_2,\,\pi)}.
\]
Hence in every such family any element other than the identity is a sticky product of two
others, and any element has roots of every natural degree. The normalized inversion number of
the copula $C^\rho_\beta$ equals $\beta/\pi$. For the
Archimedean measure $\Arch$ the family $C^{\Arch}$ consists of the limit shapes of the
permutations of the uniformly random sorting network, found by Angel et al.\ and Dauvergne. For
the Gaussian measure $\cN$ the family $C^{\cN}$ consists of Gaussian copulas and contains
$\Unif$, the uniform measure on the square, that is, the limit shape of the uniformly random
permutation; its square root within this family is the Gaussian copula with correlation
$-1/\sqrt2$. By the Herschel--Maxwell theorem the Gaussian measure is the only
rotation-invariant measure whose family contains $\Unif$; nevertheless, $\Unif$ also has
another, singular, sticky square root. A discrete analogue of the Gaussian family is the
geometric sorting network generated by rotating independent Gaussian points: after rotation by
the angle $\pi/2$ its permutation is exactly uniform.
\end{abstract}

\tableofcontents

\section{Introduction}\label{sec:intro}

\emph{Permutons} are probability measures on the unit square with uniform marginals, the limit
shapes of sequences of permutations~\cite{HKMRS2013}: a permutation $\sigma\in S_n$ corresponds
to the measure spread uniformly over the squares of side $1/n$ at the positions of the ones of
its matrix, and a permuton is a weak limit of such measures as $n\to\infty$.

A \emph{reduced word} of a permutation $\sigma\in S_n$ is a shortest expression of $\sigma$ as
a product of adjacent transpositions, that is, a shortest path from the identity permutation to
$\sigma$ in the Cayley graph of $S_n$. A reduced word of the reverse identity permutation
$w_0=(n,\dots,1)$ is called a \emph{sorting network}: applying its $\binom n2$ transpositions as
comparators, that is, swapping only when out of order, sorts any sequence; it is convenient to
say that the $k$-th transposition occurs at time $s=k/\binom n2\in[0,1]$. Angel et
al.~\cite{AHRV2007} considered uniformly random sorting networks of a given size $n$ and
conjectured, and Dauvergne~\cite{Dauvergne2022} proved, that the limit shape of the network's permutation at time $s$ is the
joint distribution of two projections of the \emph{Archimedean measure}, the uniform measure on
the sphere $S^2$ projected onto the equatorial disk, onto directions at angle $\pi s$. Such a
pair of projections lies in the square $[-1,1]^2$; after an affine change of coordinates to
$[0,1]^2$ one obtains a permuton that, as $s$ runs from $0$ to $1$, morphs from the diagonal
into the antidiagonal through ellipses inscribed in the square (Fig.~\ref{fig:families}, bottom
row).

In the $0$-Hecke monoid permutations are multiplied by the \emph{Demazure product}: a reduced
word of the second factor is applied to the first, and any transposition that would decrease
the inversion number is skipped. In the language of comparator networks this is the
concatenation of two networks, where the output of the first is fed into the second. In the
language of braids it is the concatenation of two braids in which a repeated crossing of a pair
of strands is skipped, so that each pair of strands crosses at most once. The identity
element of the monoid is the identity permutation, and the absorbing element is the reverse
identity permutation.

The \emph{sticky product} $\boxdot$ of permutons is a continuous analogue of the Demazure
product: to the limit shapes of two sequences of permutations it assigns the limit shape of
their Demazure products. It was introduced in the first author's bachelor's
thesis~\cite{PresnyakovThesis} (with the second author as academic advisor), see
also~\cite{PresnyakovTiskin}; independently and in different coordinates, the same product was
constructed by Defant~\cite{Defant2025}. For copulas, the operation $\boxdot$ is given by the tropical
$(\min,+)$ product $\odot$ (Definition~\ref{def:box}), and the identity and absorbing elements
are the Fr\'echet bounds $W$ and $M$, the copulas of the permutons supported on the
antidiagonal and on the diagonal of the square\footnote{Here a permutation is identified with
its matrix, row $i$ from top to bottom (\S\ref{sec:defs}). The identification with the graph of
the permutation, adopted in~\cite{AHRV2007} and in \S\ref{sec:rsn}, differs by the reflection
$v\mapsto1-v$ and swaps the diagonal with the antidiagonal; see \S\ref{sec:cors}.}.

The results of the paper are as follows.
\begin{enumerate}
  \item[(1)] \emph{Angle addition.} Let $\rho$ be an absolutely continuous rotation-invariant
    probability measure on the plane, and let $C^\rho=\{C^\rho_\beta\}_{\beta\in[0,\pi]}$ be the
    family of copulas of a pair of projections of a random vector distributed according to
    $\rho$ onto directions at angle $\pi-\beta$. Then $C^\rho_0=W$, $C^\rho_\pi=M$, and
    (Theorem~\ref{thm:main})
    \[
      C^\rho_{\beta_1}\odot C^\rho_{\beta_2}=C^\rho_{\min(\beta_1+\beta_2,\,\pi)} .
    \]
    We call a family with this law a \emph{geodesic} (Definition~\ref{def:geo}): it is a
    path from the identity element to the absorbing element along which the sticky product adds
    the parameter. The parameter has a meaning: $\beta/\pi$ is the normalized inversion number
    (Lemma~\ref{lem:kendall}).
  \item[(2)] \emph{The Archimedean family.} Let $\Arch$ be the Archimedean measure, the
    projection of the uniform measure on the sphere onto the equatorial disk. By the result of
    Angel et al.\ and Dauvergne, the copulas $C^{\Arch}_\beta$ are the limit shapes of the
    permutation of the uniformly random sorting network at times $s=1-\beta/\pi$. Our result is
    that the family $C^{\Arch}$ is a geodesic, that is, the limit shapes $\mu_s$ satisfy
    $\mu_{s_1}\boxdot\mu_{s_2}=\mu_{\max(s_1+s_2-1,\,0)}$ (Corollary~\ref{cor:rsn}).
  \item[(3)] \emph{The Gaussian family.} For the Gaussian measure the geodesic consists of
    Gaussian copulas with correlation $-\cos\beta$ and contains $\Unif=C^{\cN}_{\pi/2}$, the
    limit shape of the uniformly random permutation (Example~\ref{ex:families}). As in every
    geodesic, each element here has roots of every natural degree; the $n$-th root of $\Unif$
    is the Gaussian copula with correlation $-\cos(\pi/(2n))$, in particular
    $\sqrt{\Unif}:=\Unif^{\boxdot1/2}$ corresponds to correlation $-1/\sqrt2$
    (Corollary~\ref{cor:unifroots}). Geometric sorting networks, generated by rotating a set of
    points in the plane, are known~\cite{AHRV2007,Dauvergne2022}; for points with the
    Archimedean distribution Dauvergne showed that they describe the limits of subnetworks of
    the uniformly random network. We observe that for $n$ independent Gaussian points the
    permutation of such a network after rotation by the angle $\pi/2$ is exactly uniform, and
    the limit shapes of its permutations form the Gaussian geodesic
    (Proposition~\ref{prop:geom}).
  \item[(4)] \emph{Rigidity and non-uniqueness.} Among projection geodesics only the Gaussian
    one contains $\Unif$ (Theorem~\ref{thm:rigidity}, a consequence of the Herschel--Maxwell
    theorem), but $\Unif$ also has a singular sticky square root
    (Theorem~\ref{thm:nonuniqueness}), and there is a continuum of distinct geodesics
    (Proposition~\ref{prop:continuum}).
\end{enumerate}

The structure of the paper is as follows: \S\ref{sec:defs} --- copulas, $\odot$, and
geodesics; \S\ref{sec:family} --- the projection family and examples; \S\ref{sec:main} --- the
angle-addition theorem, roots, the inversion number; \S\S\ref{sec:gauss}--\ref{sec:nonuniqueness}
--- the Gaussian family, rigidity, the singular root; \S\ref{sec:rsn} --- sorting networks;
\S\ref{sec:context} --- context and open questions.

\section{Copulas and the sticky product}\label{sec:defs}

\begin{definition}[permuton, copula]\label{def:perm}
A \emph{permuton} $P$ is a probability measure on $[0,1]^2$ with uniform marginals:
$P(S\times[0,1])=P([0,1]\times S)=\Leb(S)$ for every Borel set $S\subseteq[0,1]$. Its
\emph{copula} is the distribution function
$C_P(u,v):=P([0,u]\times[0,v])$, $(u,v)\in[0,1]^2$.
\end{definition}

A copula is continuous, nondecreasing in each argument, vanishes when $u=0$ or $v=0$, satisfies
$C_P(u,1)=u$, $C_P(1,v)=v$, and lies between the Fr\'echet bounds
\begin{equation}\label{eq:frechet}
  W(u,v):=\max(u+v-1,0)\ \le\ C_P(u,v)\ \le\ \min(u,v)=:M(u,v).
\end{equation}
The increments of a copula are controlled by the marginals: for $u\le u'$, $v\le v'$
\begin{equation}\label{eq:incr}
  0\le C_P(u,v')-C_P(u,v)\le v'-v,\qquad 0\le C_P(u',v)-C_P(u,v)\le u'-u .
\end{equation}
A distribution function determines the measure uniquely, so $C_P\equiv C_Q$ implies $P=Q$.
Besides $W$ (the antidiagonal copula) and $M$ (the diagonal one) we shall need the independence
copula $\Unif(u,v)=uv$, the distribution function of Lebesgue measure on the square. Put
$P\preceq Q$ if $C_P\le C_Q$ pointwise; this order on permutons corresponds to the \emph{Bruhat
order} on permutations: by the classical criterion going back to Ehresmann~\cite{Ehresmann1934}
(see~\cite[Thm.~2.1.5]{BjornerBrenti2005}), two permutations are comparable in the Bruhat order
if and only if their count matrices $\#\{k\le i:\sigma(k)\ge j\}$ are comparable entrywise,
that is, the copulas $C_{P_\sigma}$ are comparable at the grid
nodes~\cite{PresnyakovThesis,PresnyakovTiskin}\footnote{In copula theory the same order is called the
concordance ordering~\cite{Nelsen2006}. Whether it agrees with the Bruhat order on permutations
or with its reverse depends on how a permutation is identified with a permuton; this is
immaterial for what follows.}; by~\eqref{eq:frechet}, $W$ is the least element and $M$ the
greatest.

\begin{definition}[tropical and sticky products,~\cite{PresnyakovThesis,PresnyakovTiskin}]\label{def:box}
The \emph{tropical}, or $(\min,+)$, \emph{product} of copulas is
\begin{equation}\label{eq:boxdot}
  (C_A\odot C_B)(u,v):=\min_{t\in[0,1]}\big[\,C_A(t,v)+C_B(u,1-t)\,\big],\qquad (u,v)\in[0,1]^2 .
\end{equation}
The \emph{sticky product} of permutons $P\boxdot Q$ is the permuton with copula $C_P\odot C_Q$.
\end{definition}

The minimum is attained, since the function being minimized is continuous in $t$; for the same
reason $C_A\odot C_B$ is continuous in $(u,v)$. That $C_P\odot C_Q$ is always the copula of a
permuton is shown below (Corollary~\ref{cor:limit}); this is not used in the proofs, since every
product is exhibited explicitly.

\emph{Convention.} We identify a
permuton with its copula and denote them by the same letter: below $W$, $M$, $\Unif$, and
$C^\rho_\beta$ stand both for the copula and for the permuton with this copula. In particular,
$\odot$ and $\boxdot$ are one operation in two notations, $C_{P\boxdot Q}=C_P\odot C_Q$, and
the inversion number and the support of a copula are those of the corresponding permuton.

\begin{definition}[geodesic]\label{def:geo}
A \emph{$\odot$-geodesic} is a family of copulas $\{C_\beta\}_{\beta\in[0,\pi]}$ with $C_0=W$,
$C_\pi=M$, and
\[
  C_{\beta_1}\odot C_{\beta_2}=C_{\min(\beta_1+\beta_2,\,\pi)}\qquad(\beta_1,\beta_2\in[0,\pi]).
\]
A family of permutons with these copulas is a \emph{$\boxdot$-geodesic}. Since this is one and
the same family written in two languages, from now on we simply say \emph{geodesic}; the same
applies to $\odot$- and $\boxdot$-roots.
\end{definition}

A geodesic is the image of a homomorphism from the monoid
$([0,\pi],\ (\beta_1,\beta_2)\mapsto\min(\beta_1+\beta_2,\pi))$ to the monoid of copulas; the
parameter $\beta/\pi$ plays the role of an additive logarithm, the identity element $0$ is
mapped to $W$ and the absorbing element $\pi$ to $M$.

\begin{remark}[why ``geodesic'']\label{rem:whygeo}
The name is justified by three observations.
(i) A sorting network is a shortest path from the identity permutation to the reverse identity
in the Cayley graph of $S_n$, and a path from the identity in the Cayley graph is shortest if and only if the
length of the permutation grows by one at each step. The length of a permutation is its
inversion number, and along a geodesic $C^\rho$ the normalized inversion number grows linearly
(Lemma~\ref{lem:kendall}) and adds under multiplication until the sum reaches one
(Theorem~\ref{thm:main}). Moreover, the uniformly random sorting
network, embedded in the sphere $S^{n-2}$, is close to an arc of a great
circle~\cite{AHRV2007,Dauvergne2022}.
(ii) By the angle-addition law, starting from the identity element $W$ and choosing a
``direction'' $\rho$, we move towards the absorbing element $M$ additively, as along a
one-parameter subgroup.
(iii) For the Archimedean measure the path $C^{\Arch}$ is the unique minimal-energy path between $M$
and $W$ in the space of permutons with the Wasserstein distance~\cite{RVV}, and in Riemannian
geometry minimal-energy paths are exactly the geodesics; whether something
similar holds for other profiles is unknown (\S\ref{sec:context}).
\end{remark}

\begin{lemma}[identity and absorbing element]\label{lem:basic}
For every copula $C$:
\begin{enumerate}
  \item[\rm(i)] $W\odot C=C\odot W=C$;
  \item[\rm(ii)] $M\odot C=C\odot M=M$;
  \item[\rm(iii)] $\Unif\odot\Unif=M$.
\end{enumerate}
\end{lemma}

\begin{proof}
(i) Let $\Psi(t)=\max(t+v-1,0)+C(u,1-t)$. For $t=1-v$ we have $\Psi=C(u,v)$. For $t\le1-v$
the first term vanishes and $C(u,1-t)\ge C(u,v)$. For $t>1-v$, by~\eqref{eq:incr},
$C(u,v)-C(u,1-t)\le t+v-1$, whence $\Psi(t)\ge C(u,v)$. Thus $W\odot C=C$. Similarly,
$(C\odot W)(u,v)=\min_t[C(t,v)+\max(u-t,0)]$: for $t=u$ the value is $C(u,v)$, for $t\ge u$ the
first term is $\ge C(u,v)$, and for $t<u$, by~\eqref{eq:incr}, $C(u,v)-C(t,v)\le u-t$.

(ii) $(M\odot C)(u,v)=\min_t[\min(t,v)+C(u,1-t)]$. For $t=0$ this equals $u$, for $t=1$ it equals
$v$, so the minimum is $\le\min(u,v)$. From below: $C(u,1-t)\ge W(u,1-t)=\max(u-t,0)$; for $t\ge u$
$\min(t,v)\ge\min(u,v)$, and for $t<u$ $\min(t,v)+u-t=\min(u,\,v+u-t)\ge\min(u,v)$. Symmetrically,
$(C\odot M)(u,v)=\min_t[C(t,v)+\min(u,1-t)]$: for $t=0$ this equals $u$, for $t=1$ it equals $v$;
from below, $C(t,v)\ge\max(t+v-1,0)$, and for $t\ge1-u$ we get $\ge(t+v-1)+(1-t)=v$, while for
$t<1-u$ we get $\ge u$.

(iii) $\min_t[tv+u(1-t)]=u+\min_t t(v-u)=\min(u,v)$.
\end{proof}

\paragraph{Discrete correspondence.} To a permutation $\sigma\in S_n$ we associate the permuton
$P_\sigma$: a square of side $1/n$ with density $n$ in place of each unit entry of the
permutation matrix, that is, on
$[\tfrac{i-1}n,\tfrac in]\times[1-\tfrac{\sigma(i)}n,\,1-\tfrac{\sigma(i)-1}n]$,
$i=1,\dots,n$ (row $i$ counted from the top). For a permuton $\mu$, a \emph{$\mu$-random
permutation} of size $n$ is the permutation $\sigma$ built from $n$ independent points with
distribution $\mu$: the point with the $i$-th smallest abscissa has the $(n+1-\sigma(i))$-th
smallest ordinate, so that the squares of $P_\sigma$ are the cells of the rank grid occupied by
the sample. The \emph{Demazure product} $\sigma\star\tau$ is the result of applying to $\sigma$,
in succession, the transpositions of a reduced word of $\tau$ with the rule ``a transposition
that decreases the inversion number is not performed''; this is the product in the $0$-Hecke
monoid. See Figure~\ref{fig:perm}.

\begin{figure}[ht]\centering
\begin{tikzpicture}[scale=0.55]
  % матрица перестановки sigma=3142, строка i сверху вниз
  \begin{scope}
    \draw[step=1,gray!50,thin] (0,0) grid (4,4);
    \draw (0,0) rectangle (4,4);
    \foreach \i/\j in {1/3,2/1,3/4,4/2} {
      \node at ({\j-0.5},{4-\i+0.5}) {$1$};
    }
    \foreach \i in {1,...,4} { \node[left] at (0,{4-\i+0.5}) {\scriptsize $i=\i$}; }
    \foreach \j in {1,...,4} { \node[above] at ({\j-0.5},4) {\scriptsize $\j$}; }
    \node at (2,-0.8) {\small matrix of $\sigma$};
  \end{scope}
  \begin{scope}[xshift=8cm]
    \draw[step=1,gray!50,thin] (0,0) grid (4,4);
    \foreach \i/\j in {1/3,2/1,3/4,4/2} {
      \fill[blue!35] ({\i-1},{4-\j}) rectangle ({\i},{4-\j+1});
    }
    \draw[step=1,gray!50,thin] (0,0) grid (4,4);
    \draw (0,0) rectangle (4,4);
    \draw[dashed,red!60] (0,4) -- (4,0);
    \node[below] at (0,0) {\scriptsize $0$}; \node[below] at (4,0) {\scriptsize $1$};
    \node[left] at (0,4) {\scriptsize $1$};
    \node[below] at (2,0) {\scriptsize $u$}; \node[left] at (0,2) {\scriptsize $v$};
    \node at (2,-0.8) {\small permuton $P_\sigma$};
  \end{scope}
\end{tikzpicture}
\caption{The permuton $P_\sigma$ for $\sigma=3142$: on the left the permutation matrix (a one in row
$i$ and column $\sigma(i)$, rows from top to bottom), on the right squares of side $\tfrac14$ and
density $4$ in the same positions. The identity permutation gives squares along the antidiagonal
(dashed), the reverse identity along the diagonal.}
\label{fig:perm}
\end{figure}
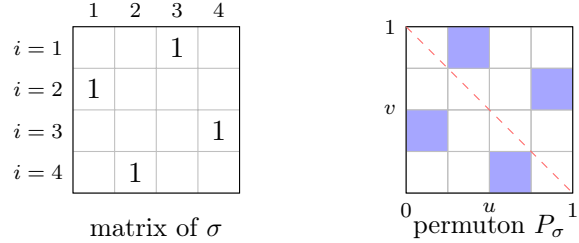

\begin{theorem}[permuton sampling theorem~\cite{HKMRS2013}]\label{thm:sampling}
If $\sigma_n$ is a $\mu$-random permutation of size $n$, then $P_{\sigma_n}\to\mu$ weakly
almost surely as $n\to\infty$.
\end{theorem}

\begin{theorem}[Tiskin~\cite{tiskin2008semi,tiskin2015fast}; independently Pflueger~\cite{Pflueger2022}]\label{thm:tiskin}
For $\sigma,\tau\in S_n$ the functions $C_{P_\sigma}\odot C_{P_\tau}$ and $C_{P_{\sigma\star\tau}}$
coincide at the grid nodes $\{(a/n,b/n)\}_{a,b=0}^n$.
\end{theorem}

This is a reformulation of Tiskin's theorem that the dominance-sum matrix of a Demazure product
is the tropical product of the dominance-sum matrices of the factors: the values of $C_{P_\sigma}$
at the grid nodes are the normalized dominance sums of the matrix of $\sigma$, and the minimum
over $t\in[0,1]$ in~\eqref{eq:boxdot} is attained at grid nodes, since on each interval
$[\tfrac{j-1}n,\tfrac jn]$ both terms are linear in $t$.

\begin{lemma}[Lipschitz property of $\odot$]\label{lem:lip}
For any copulas, $\|A\odot B-A'\odot B'\|_\infty\le\|A-A'\|_\infty+\|B-B'\|_\infty$, and
$A\odot B$ is $1$-Lipschitz in each argument. In particular, $\odot$ is continuous with respect
to weak convergence of permutons.
\end{lemma}

\begin{proof}
For each $t$ the expression under the minimum in~\eqref{eq:boxdot} changes by at most
$\|A-A'\|_\infty+\|B-B'\|_\infty$, hence so does the minimum over $t$; the Lipschitz property
in $u$ and in $v$ is inherited from $B$ and $A$ in the same way. Copulas are $1$-Lipschitz in
each argument, so their pointwise convergence is equivalent to uniform convergence, and for
permutons pointwise convergence of the distribution functions is weak convergence.
\end{proof}

\begin{corollary}[limit of Demazure products]\label{cor:limit}
If $u_n,v_n\in S_n$ and $P_{u_n}\to A$, $P_{v_n}\to B$ weakly, then $P_{u_n\star v_n}$ converges
weakly to the permuton with copula $C_A\odot C_B$. In particular, $C_A\odot C_B$ is the copula
of a permuton, and $\boxdot$ is defined on all permutons.
\end{corollary}

\begin{proof}
The functions $C_{P_{u_n}}\odot C_{P_{v_n}}$ and $C_{P_{u_n\star v_n}}$ are $1$-Lipschitz in each
argument (Lemma~\ref{lem:lip}) and coincide at the nodes of the grid with spacing $1/n$
(Theorem~\ref{thm:tiskin}), hence they differ by at most $2/n$; by Lemma~\ref{lem:lip} the first
of them converges uniformly to $C_A\odot C_B$. A weak limit of permutons is a permuton, since
uniformity of the marginals is preserved.
\end{proof}

\begin{remark}[monoid]\label{rem:monoid}
By Corollary~\ref{cor:limit} and Theorem~\ref{thm:sampling}, $\boxdot$ is an operation on all
permutons, and it is associative because the Demazure product is; together with
Lemma~\ref{lem:basic} this makes it a monoid with identity element $W$ and absorbing element
$M$, a continuous analogue of the $0$-Hecke monoid~\cite{PresnyakovThesis,PresnyakovTiskin,Defant2025}. In
general $\boxdot$ is not commutative. The notations $\star$ and $\boxdot$, like the names
``Demazure product'' and ``sticky product'', refer to essentially the same operation; we write
$\star$ for permutations and $\boxdot$ for permutons, and the difference in names is historical.
Neither closedness nor associativity is used below: every product is computed explicitly, and
within a geodesic the product reduces to addition of parameters.
\end{remark}

\section{The projection family $C^\rho$}\label{sec:family}

Throughout, $\rho$, the \emph{profile} of the family, is an \emph{absolutely continuous} ($\rho\ll\Leb$) \emph{rotation-invariant}
(invariant under $SO(2)$) probability measure on $\R^2$, $G=(G_1,G_2)\sim\rho$,
$u_\theta=(\cos\theta,\sin\theta)$, and $X_\theta:=\langle G,u_\theta\rangle$. In polar
coordinates $G=R\,(\cos\Theta,\sin\Theta)$, with $R>0$ almost surely.

\begin{lemma}[projections]\label{lem:proj}
\begin{enumerate}
  \item[\rm(i)] All $X_\theta$ have the same distribution function $F$; it is continuous and
    symmetric, $F(-t)=1-F(t)$. In particular $F^{-1}(\tfrac12)=0$ and $F^{-1}(1-p)=-F^{-1}(p)$,
    where $F^{-1}$ is the quantile function.
  \item[\rm(ii)] The joint distribution of the pair $(X_\theta,X_{\theta'})$ depends only on
    $\theta'-\theta$.
  \item[\rm(iii)] The direction $\Theta$ is uniformly distributed on $[0,2\pi)$.
\end{enumerate}
\end{lemma}

\begin{proof}
For a rotation $R_\omega$ we have $X_{\theta+\omega}=\langle R_{-\omega}G,u_\theta\rangle$ and
$R_{-\omega}G\eqd G$, hence $(X_{\theta+\omega},X_{\theta'+\omega})\eqd(X_\theta,X_{\theta'})$
--- this is (ii), and for $\theta'=\theta$ it gives the equidistribution of all $X_\theta$. From
$u_{\theta+\pi}=-u_\theta$ it follows that $-X_\theta\eqd X_\theta$, that is, $F$ is symmetric.
The projection of an absolutely continuous measure onto a line is absolutely continuous, so $F$ is
continuous; moreover, for each $R=r>0$ the variable $r\cos\Theta$ has a positive density on
$(-r,r)$, so that $F$ is strictly increasing wherever $0<F<1$, and $F^{-1}(\tfrac12)=0$. Finally,
the distribution of $\Theta$ is a probability measure on the circle invariant under all rotations,
that is, the uniform one.
\end{proof}

\begin{definition}[projection copula]\label{def:proj}
For $\beta\in[0,\pi]$ put
\[
  C_\beta^\rho(u,v):=\Prob{X_0\le F^{-1}(u),\ -X_\beta\le F^{-1}(v)},\qquad (u,v)\in[0,1]^2 .
\]
\end{definition}

This is the copula of the pair $(X_0,-X_\beta)$ --- the projections onto $u_0$ and onto the ray
$-u_\beta$, which makes the angle $\pi-\beta$ with $u_0$. The family
$\{C^\rho_\beta\}_{\beta\in[0,\pi]}$ is denoted by $C^\rho$.

\begin{remark}[invariance of the copula]\label{rem:inv}
The copula of a pair of random variables does not change under strictly increasing
transformations of each of them. In particular, the copula is unchanged under the homothety
$G\mapsto cG$, $c>0$, so the family $C^\rho$ depends on $\rho$ only up to scale.
\end{remark}

\begin{lemma}[$C_\beta^\rho$ is a permuton; endpoints of the family]\label{lem:perm}
$C_\beta^\rho$ is the copula of a permuton, $C_0^\rho=W$ and $C_\pi^\rho=M$.
\end{lemma}

\begin{proof}
By Lemma~\ref{lem:proj}(i), $-X_\beta$ has distribution function $F$, and for continuous $F$ the
events $\{X\le F^{-1}(u)\}$ and $\{F(X)\le u\}$ coincide up to a set of measure zero. Therefore
$C_\beta^\rho$ is the distribution function of the pair $(F(X_0),F(-X_\beta))$ with uniform
marginals. For $\beta=0$: $F(-X_0)=1-F(X_0)$, and
$C_0^\rho(u,v)=\Prob{1-v\le F(X_0)\le u}=\max(u+v-1,0)=W$. For $\beta=\pi$: $-X_\pi=X_0$, the pair
lies on the diagonal, and $C_\pi^\rho=\min(u,v)=M$.
\end{proof}

\begin{remark}[correlation]\label{rem:corr}
If $\rho$ has finite variance $\sigma^2$, then by rotation invariance
$\operatorname{cov}(G)=\sigma^2I_2$ and $\corr(X_0,-X_\beta)=-\cos\beta$ for \emph{all} radial
profiles. The families differ not in their correlation but in the full copula.
\end{remark}

\begin{example}[the Gaussian and Archimedean families]\label{ex:families}
(a) Let $\cN:=\cN(0,I_2)$. The pair $(X_0,-X_\beta)$ is Gaussian with standard marginals and
correlation $-\langle u_0,u_\beta\rangle=-\cos\beta$, hence by Remark~\ref{rem:inv}
\[
  C^{\cN}_\beta=C_{-\cos\beta},\qquad C_r(u,v):=\Phi_r\big(\Phi^{-1}(u),\Phi^{-1}(v)\big),
\]
where $C_r$ is the standard Gaussian copula with correlation $r\in[-1,1]$, $\Phi$ is the
standard normal distribution function, and $\Phi_r$ is the bivariate one with correlation $r$;
in the limiting cases $C_{-1}=W$, $C_0=\Unif$, $C_1=M$~\cite{Nelsen2006}. At $\beta=\pi/2$ the
correlation vanishes and $C^{\cN}_{\pi/2}=\Unif$: the Gaussian family passes through
independence exactly at its midpoint. For $\beta\in(0,\pi)$ the permuton $C^{\cN}_\beta$ has a
positive density on the whole open square.

(b) The \emph{Archimedean measure} $\Arch$ is the projection of the uniform measure on the
sphere $S^2$ onto the equatorial disk; its density on the unit disk is
$\frac{1}{2\pi\sqrt{1-a^2-b^2}}$. It is rotation-invariant and absolutely continuous (the
singularity of the density at the boundary circle is integrable), and all its projections are
uniform on $[-1,1]$ (Archimedes' theorem), so that $F(x)=(x+1)/2$ and the permuton with copula
$C^{\Arch}_\beta$ is an affine image of the distribution of the pair $(X_0,-X_\beta)$,
supported in an ellipse inscribed in the square. In \S\ref{sec:rsn} we discuss the relationship
of this family with random sorting networks.

Fig.~\ref{fig:families} shows samples from both families at
$\beta=0,\tfrac\pi4,\tfrac\pi2,\tfrac{3\pi}4,\pi$: both morph from the antidiagonal $W$ into
the diagonal $M$, but only the Gaussian one passes through the uniform distribution.
\end{example}

\begin{figure}[ht]\centering
\includegraphics[width=\textwidth]{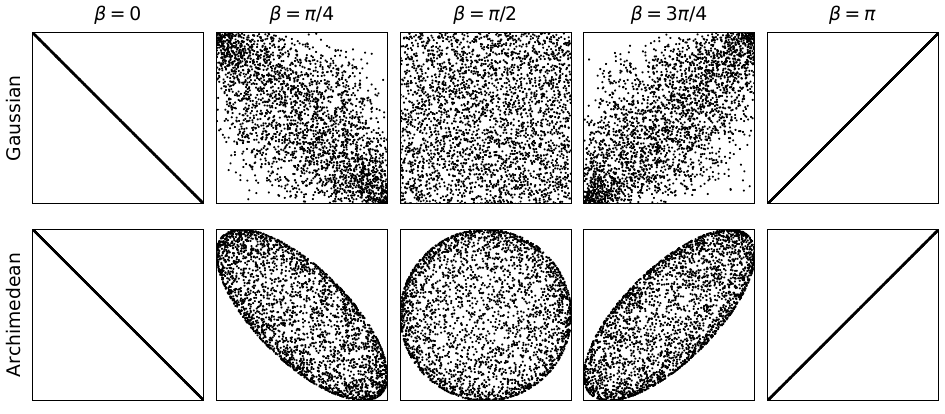}
\caption{Samples of $3000$ points from the permutons $C^{\cN}_\beta$ (top row) and
$C^{\Arch}_\beta$ (bottom row) at $\beta=0,\tfrac\pi4,\tfrac\pi2,\tfrac{3\pi}4,\pi$.
In both rows the antidiagonal morphs into the diagonal: the Gaussian family passes through the
uniform distribution at $\beta=\pi/2$, the Archimedean one stays supported in ellipses, at
$\beta=\pi/2$ in the inscribed disk. The sorting network of \S\ref{sec:rsn} traverses the
bottom row from right to left.}
\label{fig:families}
\end{figure}

\section{Angle addition}\label{sec:main}

\begin{theorem}[angle addition]\label{thm:main}
Let $\rho$ be an absolutely continuous rotation-invariant probability measure on $\R^2$.
Then for all $\beta_1,\beta_2\in[0,\pi]$
\[
  \boxed{\ C_{\beta_1}^\rho\odot C_{\beta_2}^\rho=C_{\min(\beta_1+\beta_2,\,\pi)}^\rho\ }
\]
as copulas (that is, as measures) on $[0,1]^2$.
\end{theorem}

Both sides are continuous on $[0,1]^2$, so it suffices to verify the equality at interior points
$(u,v)\in(0,1)^2$. For $\beta_m=0$ the claim is Lemma~\ref{lem:basic}(i)
($C_0^\rho=W$), for $\beta_m=\pi$ it is Lemma~\ref{lem:basic}(ii) ($C_\pi^\rho=M$). From now on
$\beta_1,\beta_2\in(0,\pi)$ and $(u,v)\in(0,1)^2$ are fixed; denote the quantiles $x:=F^{-1}(u)$,
$y:=F^{-1}(v)$; by continuity of $F$, $x\le y$ is equivalent to $u\le v$.

\emph{Plan of the proof.} In \S\ref{sec:reduction} the product is reduced to the minimization
over $s\in\R$ of a function $g(s)$, and the three events involved are realized as half-planes
determined by a single random vector $G$; three identities express $g$ through the
probabilities of the four pieces into which the line $\{B=s\}$ cuts the two angular sectors
between the other two lines. In \S\ref{sec:geometry} we determine on which side of that line
the sectors lie: for $\beta_1+\beta_2<\pi$ the minimum is attained at a single point and equals
$C^\rho_{\beta_1+\beta_2}(u,v)$, while for $\beta_1+\beta_2\ge\pi$ it is not attained: the
infimum equals $M(u,v)$ and is approached as $s\to\pm\infty$.

\subsection{Reduction to three lines}\label{sec:reduction}

We realize both factors by projections of \emph{one} vector $G\sim\rho$. Put
\begin{equation}\label{eq:triple}
  A:=X_{-\beta_2},\qquad B:=X_0,\qquad C:=-X_{\beta_1}=X_{\beta_1+\pi};
\end{equation}
all three variables have distribution function $F$. The events $\{A\le x\}$, $\{B\le s\}$,
$\{C\le y\}$ are half-planes with boundary lines $\{A=x\}$, $\{B=s\}$, $\{C=y\}$ and
outward normals $u_{-\beta_2}$, $u_0$, $u_{\beta_1+\pi}$ (Fig.~\ref{fig:wedges}).

\begin{lemma}[reduction]\label{lem:reduce}
For $s\in\R$ put
\begin{equation}\label{eq:g}
  g(s):=\Prob{B\le s,\ C\le y}+\Prob{A\le x,\ B\ge s}.
\end{equation}
Then $g(s)=C_{\beta_1}^\rho(F(s),v)+C_{\beta_2}^\rho(u,F(-s))$ for all $s$, and
\[
  (C_{\beta_1}^\rho\odot C_{\beta_2}^\rho)(u,v)=\inf_{s\in\R}g(s)\ \le\ \min(u,v).
\]
If $\beta_1+\beta_2\le\pi$, then moreover $\Prob{A\le x,\ C\le y}=C_{\beta_1+\beta_2}^\rho(u,v)$.
\end{lemma}

\begin{proof}
The pair $(B,C)=(X_0,-X_{\beta_1})$ gives $\Prob{B\le s,C\le y}=C_{\beta_1}^\rho(F(s),v)$: the
events $\{X_0\le s\}$ and $\{X_0\le F^{-1}(F(s))\}$ differ by a set of probability zero. The pair
$(A,B)=(X_{-\beta_2},X_0)\eqd(X_0,X_{\beta_2})$ (Lemma~\ref{lem:proj}(ii)) gives
$\Prob{A\le x,B\ge s}=\Prob{X_0\le x,\,-X_{\beta_2}\le-s}$, which equals
$C_{\beta_2}^\rho(u,F(-s))$, since $F(-s)=\Prob{-X_{\beta_2}\le -s}$ and $F$ is continuous. Finally,
$(A,C)=(X_{-\beta_2},-X_{\beta_1})\eqd(X_0,-X_{\beta_1+\beta_2})$, whence the last equality.

The image $F(\R)$ contains $(0,1)$, and for $t=F(s)$ we have $1-t=F(-s)$, so that the expression
under the minimum in~\eqref{eq:boxdot} equals $g(s)$. The values $t\in\{0,1\}$ give
$C_{\beta_2}^\rho(u,1)=u$ and $C_{\beta_1}^\rho(1,v)=v$; these are the limits of $g(s)$ as
$s\to\mp\infty$ (the event $\{B\le s\}$ decreases to the empty event, $\{B\ge s\}$ increases to
the sure event). Hence $\min_{t\in[0,1]}=\inf_{s\in\R}g$, and $\inf g\le\min(u,v)$.
\end{proof}

The lines $\{A=x\}$ and $\{C=y\}$ divide the plane into four angular sectors; we need the two
\emph{mixed} ones:
\[
  M_+:=\{A\le x,\ C>y\},\qquad M_-:=\{A>x,\ C\le y\}.
\]
The line $\{B=s\}$ cuts each of them in two:
\[
  R_1:=M_-\cap\{B\le s\},\quad R_4:=M_-\cap\{B>s\},\qquad
  R_2:=M_+\cap\{B<s\},\quad R_3:=M_+\cap\{B\ge s\}.
\]

\begin{lemma}[three identities]\label{lem:atoms}
For all $s$
\begin{equation}\label{eq:atomid}
\begin{aligned}
  g(s)-\Prob{A\le x,\ C\le y}&=\Prob{R_1}+\Prob{R_3},\\
  g(s)-u&=\Prob{R_1}-\Prob{R_2},\\
  g(s)-v&=\Prob{R_3}-\Prob{R_4}.
\end{aligned}
\end{equation}
\end{lemma}

\begin{proof}
The event $\{B=s\}$ has probability zero. Split the terms of~\eqref{eq:g} according to the
missing coordinate: $\Prob{B\le s,C\le y}=\Prob{A\le x,B\le s,C\le y}+\Prob{R_1}$ and
$\Prob{A\le x,B\ge s}=\Prob{A\le x,B\ge s,C\le y}+\Prob{R_3}$; the sum of the first terms equals
$\Prob{A\le x,C\le y}$. Next,
\[
  u=\Prob{A\le x}=\Prob{A\le x,B<s,C\le y}+\Prob{R_2}+\Prob{A\le x,B\ge s},
\]
and subtracting from $g(s)$ gives the second identity; the third follows symmetrically from
$v=\Prob{B\le s,C\le y}+\Prob{A\le x,B>s,C\le y}+\Prob{R_4}$.
\end{proof}

\subsection{Interior and absorption regimes}\label{sec:geometry}

\begin{lemma}[sides of the mixed sectors]\label{lem:sides}
\begin{enumerate}
  \item[\rm(a)] If $\beta_1+\beta_2<\pi$, then there exists $s^\ast\in\R$ such that
    $M_-\subset\{B>s^\ast\}$ and $M_+\subset\{B<s^\ast\}$.
  \item[\rm(b)] If $\beta_1+\beta_2>\pi$, then there exists $s^\ast\in\R$ such that
    $M_+\subset\{B>s^\ast\}$ and $M_-\subset\{B<s^\ast\}$.
  \item[\rm(c)] If $\beta_1+\beta_2=\pi$, then $A=C$, and $M_+=\varnothing$ for $x\le y$,
    $M_-=\varnothing$ for $x>y$.
\end{enumerate}
\end{lemma}

\begin{proof}
(c) For $\beta_1+\beta_2=\pi$ we have $u_{\beta_1+\pi}=u_{2\pi-\beta_2}=u_{-\beta_2}$, that is,
$C=A$; then $M_+=\{y<A\le x\}$ and $M_-=\{x<A\le y\}$.

(a), (b) For $\beta_1+\beta_2\ne\pi$ the vectors $u_{-\beta_2}$ and $u_{\beta_1}$ form a basis,
and a direct check of coordinates gives
\[
  u_0=\lambda\,u_{-\beta_2}+\mu\,u_{\beta_1},\qquad
  \lambda=\frac{\sin\beta_1}{\sin(\beta_1+\beta_2)},\quad
  \mu=\frac{\sin\beta_2}{\sin(\beta_1+\beta_2)} .
\]
Since $C=-\langle G,u_{\beta_1}\rangle$, this yields the identity $B=\lambda A-\mu C$. Put
$s^\ast:=\lambda x-\mu y$; then
\[
  B-s^\ast=\lambda\,(A-x)-\mu\,(C-y).
\]
On $M_+$ the bracket $A-x$ is nonpositive and $C-y$ is positive, so $B-s^\ast$ has the sign
opposite to that of $\lambda,\mu$; on $M_-$ it has the same sign. It remains to note that
$\lambda,\mu>0$ for $\beta_1+\beta_2<\pi$ and $\lambda,\mu<0$ for $\beta_1+\beta_2>\pi$
(the numerators are positive for $\beta_1,\beta_2\in(0,\pi)$). Geometrically, $s^\ast$ is the
value of $B$ at the intersection point $\mathbf p$ of the lines $\{A=x\}$ and $\{C=y\}$
(Fig.~\ref{fig:wedges}).
\end{proof}

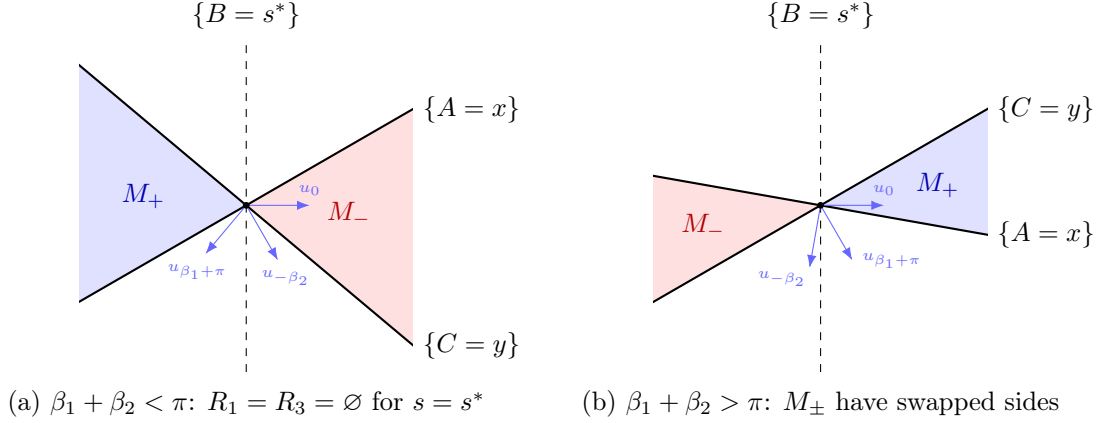
\begin{figure}[ht]\centering
\begin{tikzpicture}[scale=1.05,>=Latex]
  \begin{scope}
    \clip (-2.1,-2.1) rectangle (2.1,2.1);
    \fill[blue!12] (0,0) -- (140:5) -- (210:5) -- cycle;
    \fill[red!12]  (0,0) -- (-40:5) -- (30:5) -- cycle;
    \draw[thick] (30:3) -- (210:3);
    \draw[thick] (140:3) -- (-40:3);
    \draw[dashed] (0,-2.1) -- (0,2.1);
  \end{scope}
  \node[right] at (2.1,1.21) {\small $\{A=x\}$};
  \node[right] at (2.1,-1.76) {\small $\{C=y\}$};
  \node[above] at (0,2.1) {\small $\{B=s^\ast\}$};
  \node[blue!70!black] at (175:1.3) {\small $M_+$};
  \node[red!70!black] at (-5:1.3) {\small $M_-$};
  \draw[->,blue!60] (0,0) -- (-60:0.8) node[below,xshift=3pt]{\tiny $u_{-\beta_2}$};
  \draw[->,blue!60] (0,0) -- (230:0.8) node[below,xshift=-3pt]{\tiny $u_{\beta_1+\pi}$};
  \draw[->,blue!60] (0,0) -- (0:0.8) node[above]{\tiny $u_0$};
  \fill (0,0) circle (1.2pt);
  \node at (0,-2.5) {\small (a) $\beta_1+\beta_2<\pi$: $R_1=R_3=\varnothing$ for $s=s^\ast$};
\end{tikzpicture}
\hspace{3mm}
\begin{tikzpicture}[scale=1.05,>=Latex]
  \begin{scope}
    \clip (-2.1,-2.1) rectangle (2.1,2.1);
    \fill[blue!12] (0,0) -- (-10:5) -- (30:5) -- cycle;
    \fill[red!12]  (0,0) -- (170:5) -- (210:5) -- cycle;
    \draw[thick] (-10:3) -- (170:3);
    \draw[thick] (30:3) -- (210:3);
    \draw[dashed] (0,-2.1) -- (0,2.1);
  \end{scope}
  \node[right] at (2.1,-0.37) {\small $\{A=x\}$};
  \node[right] at (2.1,1.21) {\small $\{C=y\}$};
  \node[above] at (0,2.1) {\small $\{B=s^\ast\}$};
  \node[blue!70!black] at (10:1.5) {\small $M_+$};
  \node[red!70!black] at (190:1.5) {\small $M_-$};
  \draw[->,blue!60] (0,0) -- (-100:0.8) node[below left,inner sep=1pt]{\tiny $u_{-\beta_2}$};
  \draw[->,blue!60] (0,0) -- (300:0.8) node[right]{\tiny $u_{\beta_1+\pi}$};
  \draw[->,blue!60] (0,0) -- (0:0.8) node[above]{\tiny $u_0$};
  \fill (0,0) circle (1.2pt);
  \node at (0,-2.5) {\small (b) $\beta_1+\beta_2>\pi$: $M_\pm$ have swapped sides};
\end{tikzpicture}
\caption{The mixed sectors $M_+=\{A\le x,C>y\}$ and $M_-=\{A>x,C\le y\}$ and the vertical line
$\{B=s^\ast\}$ through the point $\mathbf p$. In the interior regime (a) both ``extra'' pieces
$R_1=M_-\cap\{B\le s^\ast\}$ and $R_3=M_+\cap\{B\ge s^\ast\}$ are empty. In the absorption
regime (b), for every $s$ one of the sectors lies entirely on the ``right'' side:
$R_2=\varnothing$ for $s\le s^\ast$, $R_4=\varnothing$ for $s\ge s^\ast$.}
\label{fig:wedges}
\end{figure}

\begin{proof}[Proof of Theorem~\ref{thm:main}]
The cases $\beta_m\in\{0,\pi\}$ were settled right after the statement; let
$\beta_1,\beta_2\in(0,\pi)$.

\emph{Interior regime} $\beta_1+\beta_2<\pi$. Take $s=s^\ast$. By Lemma~\ref{lem:sides}(a),
$R_1\subset M_-\cap\{B\le s^\ast\}=\varnothing$ and $R_3\subset M_+\cap\{B\ge s^\ast\}=\varnothing$,
so that by the first identity in~\eqref{eq:atomid}
$g(s^\ast)=\Prob{A\le x,C\le y}=C_{\beta_1+\beta_2}^\rho(u,v)$
(Lemma~\ref{lem:reduce}). The same identity gives $g(s)\ge\Prob{A\le x,C\le y}$ for all $s$,
hence $\inf g=g(s^\ast)$, and by Lemma~\ref{lem:reduce}
$(C_{\beta_1}^\rho\odot C_{\beta_2}^\rho)(u,v)=C_{\beta_1+\beta_2}^\rho(u,v)$.

\emph{Absorption regime} $\beta_1+\beta_2>\pi$. For $s\le s^\ast$, by Lemma~\ref{lem:sides}(b),
$R_2\subset M_+\cap\{B<s\}\subset\{B>s^\ast\}\cap\{B<s\}=\varnothing$, and the second
identity in~\eqref{eq:atomid} gives $g(s)\ge u$. For $s\ge s^\ast$, similarly, $R_4=\varnothing$
and $g(s)\ge v$. Thus $g\ge\min(u,v)$ everywhere, while $\inf g\le\min(u,v)$ by
Lemma~\ref{lem:reduce}; hence
$(C_{\beta_1}^\rho\odot C_{\beta_2}^\rho)(u,v)=\min(u,v)=M(u,v)=C_\pi^\rho(u,v)$.

\emph{Boundary case} $\beta_1+\beta_2=\pi$. By Lemma~\ref{lem:sides}(c), for $x\le y$ (that is,
$u\le v$) $M_+$ is empty, whence $R_2=\varnothing$ and $g\ge u=\min(u,v)$ for all $s$; for $x>y$
$M_-$ is empty, $R_4=\varnothing$ and $g\ge v=\min(u,v)$. As above, $\inf g=\min(u,v)=M(u,v)$.
\end{proof}

\begin{remark}[what was used]\label{rem:hyp}
Of the absolute continuity of $\rho$, the proof used only the continuity of $F$
(needed so that the event $\{B=s\}$ has probability zero); for a rotation-invariant
measure this is equivalent to $\rho(\{0\})=0$, since $X_0=R\cos\Theta$ with $\Theta$ uniform, and
$F$ can have an atom only if $R=0$. The coincidence of the lines in case (c) requires
no measure-theoretic considerations: one of the sectors is simply empty. Absolute continuity is
essentially needed only in the rigidity theorem (\S\ref{sec:rigidity}).
\end{remark}

\subsection{Injectivity, roots, inversion number}\label{sec:cors}

By Theorem~\ref{thm:main} and Lemma~\ref{lem:perm}, the family $C^\rho$ is a geodesic
(Definition~\ref{def:geo}). We show that the map $\beta\mapsto C^\rho_\beta$ is injective,
describe the roots, and give the parameter $\beta/\pi$ a probabilistic meaning.

\begin{lemma}[value at the square's center; injectivity]\label{lem:center}
$C_\beta^\rho(\tfrac12,\tfrac12)=\beta/2\pi$. In particular, $\beta\mapsto C_\beta^\rho$ is injective.
\end{lemma}

\begin{proof}
$F^{-1}(\tfrac12)=0$, hence
\[
  C_\beta^\rho(\tfrac12,\tfrac12)=\Prob{X_0\le0,\ X_\beta\ge0}
  =\Prob{\cos\Theta\le0,\ \cos(\Theta-\beta)\ge0}.
\]
This is the arc $\Theta\in[\tfrac\pi2,\tfrac\pi2+\beta]$ of opening $\beta$, whose probability is
$\beta/2\pi$ (Lemma~\ref{lem:proj}(iii)).
\end{proof}

\begin{corollary}[roots within the family]\label{cor:roots}
For $\beta\in(0,\pi)$ and $n\ge1$, the copula $C_\beta^\rho$ has a unique $n$-th $\odot$-root
within the family $C^\rho$, namely $C_{\beta/n}^\rho$.
\end{corollary}

\begin{proof}
$(C_{\beta/n}^\rho)^{\odot n}=C_\beta^\rho$ by Theorem~\ref{thm:main}, since all partial sums
$k\beta/n<\pi$. If $(C_{\beta'}^\rho)^{\odot n}=C_\beta^\rho$, then
$C_{\min(n\beta',\pi)}^\rho=C_\beta^\rho$, and by injectivity $n\beta'=\beta$. (For $\beta=\pi$
uniqueness is lost: $M=(C_{\beta'}^\rho)^{\odot n}$ for all $\beta'\ge\pi/n$.)
\end{proof}

\begin{lemma}[direct sums]\label{lem:multi}
Let $0=a_0<a_1<\dots<a_m=1$, $w_k:=a_k-a_{k-1}$, $I_k:=[a_{k-1},a_k]$, and
$J_k:=[1-a_k,\,1-a_{k-1}]$, $k=1,\dots,m$; the blocks $I_k\times J_k$ are arranged along the
antidiagonal (Figure~\ref{fig:blocks}). For copulas $A_1,\dots,A_m$ denote by $\bigoplus_kA_k$ the copula of the permuton
that carries, in the block $I_k\times J_k$, mass $w_k$ distributed as a copy of the permuton
of $A_k$ scaled down by the factor $w_k$. Then for any copulas $A_k,B_k$
\[
  \Big(\bigoplus_kA_k\Big)\odot\Big(\bigoplus_kB_k\Big)=\bigoplus_k\,(A_k\odot B_k).
\]
\end{lemma}

\begin{proof}
For $\sigma_k\in S_{n_k}$ denote by $\sigma_1\oplus\dots\oplus\sigma_m\in S_n$,
$n=n_1+\dots+n_m$, the permutation acting as $\sigma_k$ on the $k$-th block of positions; the
blocks of its matrix lie along the antidiagonal, so the copula of
$P_{\sigma_1\oplus\dots\oplus\sigma_m}$ is $\bigoplus_kC_{P_{\sigma_k}}$ with the partition
$a_k=(n_1+\dots+n_k)/n$. The Demazure product of block permutations is computed blockwise,
\[
  (\sigma_1\oplus\dots\oplus\sigma_m)\star(\tau_1\oplus\dots\oplus\tau_m)
  =(\sigma_1\star\tau_1)\oplus\dots\oplus(\sigma_m\star\tau_m),
\]
since the reduced words of the factors contain no generators connecting different blocks, and
generators of different blocks commute.

Let $n\to\infty$, $n_k/n\to w_k$, and let $\sigma_k^{(n)},\tau_k^{(n)}\in S_{n_k}$ be
$\mu$-random permutations for the permutons with copulas $A_k$ and $B_k$; by
Theorem~\ref{thm:sampling}, almost surely $P_{\sigma_k^{(n)}}\to A_k$ and
$P_{\tau_k^{(n)}}\to B_k$. Put $u_n=\bigoplus_k\sigma_k^{(n)}$, $v_n=\bigoplus_k\tau_k^{(n)}$.
The operation $\oplus$ is continuous in the blocks and in the partition, so the copulas of
$P_{u_n}$ and $P_{v_n}$ converge to $\bigoplus_kA_k$ and $\bigoplus_kB_k$, and by
Corollary~\ref{cor:limit} the copulas of $P_{u_n\star v_n}$ converge to
$(\bigoplus_kA_k)\odot(\bigoplus_kB_k)$. On the other hand,
$u_n\star v_n=\bigoplus_k(\sigma_k^{(n)}\star\tau_k^{(n)})$, and by the same corollary applied in
each block, rescaled to the unit square, the copulas of $P_{u_n\star v_n}$ converge to
$\bigoplus_k(A_k\odot B_k)$. The limit is unique.
\end{proof}

\begin{remark}[multi-parameter submonoids]\label{rem:multi}
For a fixed partition $\{a_k\}$ and profiles $\rho_1,\dots,\rho_m$, the copulas
$\bigoplus_kC^{\rho_k}_{\beta_k}$ form, by Lemma~\ref{lem:multi}, an $m$-parameter submonoid in
which angles add coordinatewise and roots are unique whenever all $\beta_k<\pi$ (Figure~\ref{fig:blocks}); this is a
continuous analogue of the parabolic submonoids
$\mathcal H_{n_1}\times\dots\times\mathcal H_{n_m}$.
\end{remark}

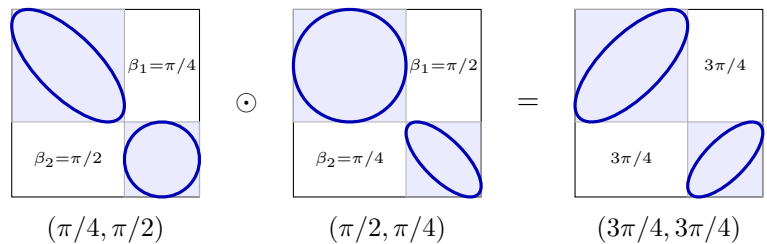
\begin{figure}[ht]\centering
\begin{tikzpicture}[scale=0.62]

  \begin{scope}[xshift=0cm]
    \fill[blue!8] (0,1.6) rectangle (2.4,4);
    \fill[blue!8] (2.4,0) rectangle (4,1.6);
    \draw (0,0) rectangle (4,4);
    \draw[gray!60] (0,1.6) rectangle (2.4,4);
    \draw[gray!60] (2.4,0) rectangle (4,1.6);
    \draw[very thick,blue!70!black,rotate around={-45:(1.200,2.800)}] (1.200,2.800) ellipse (1.568 and 0.649);
    \node at (3.200,2.800) {\tiny $\beta_1{=}\pi/4$};
    \draw[very thick,blue!70!black,rotate around={-45:(3.200,0.800)}] (3.200,0.800) ellipse (0.800 and 0.800);
    \node at (1.200,0.800) {\tiny $\beta_2{=}\pi/2$};
    \node at (2,-0.7) {\small $(\pi/4,\pi/2)$};
  \end{scope}

  \begin{scope}[xshift=6cm]
    \fill[blue!8] (0,1.6) rectangle (2.4,4);
    \fill[blue!8] (2.4,0) rectangle (4,1.6);
    \draw (0,0) rectangle (4,4);
    \draw[gray!60] (0,1.6) rectangle (2.4,4);
    \draw[gray!60] (2.4,0) rectangle (4,1.6);
    \draw[very thick,blue!70!black,rotate around={-45:(1.200,2.800)}] (1.200,2.800) ellipse (1.200 and 1.200);
    \node at (3.200,2.800) {\tiny $\beta_1{=}\pi/2$};
    \draw[very thick,blue!70!black,rotate around={-45:(3.200,0.800)}] (3.200,0.800) ellipse (1.045 and 0.433);
    \node at (1.200,0.800) {\tiny $\beta_2{=}\pi/4$};
    \node at (2,-0.7) {\small $(\pi/2,\pi/4)$};
  \end{scope}

  \begin{scope}[xshift=12cm]
    \fill[blue!8] (0,1.6) rectangle (2.4,4);
    \fill[blue!8] (2.4,0) rectangle (4,1.6);
    \draw (0,0) rectangle (4,4);
    \draw[gray!60] (0,1.6) rectangle (2.4,4);
    \draw[gray!60] (2.4,0) rectangle (4,1.6);
    \draw[very thick,blue!70!black,rotate around={-45:(1.200,2.800)}] (1.200,2.800) ellipse (0.649 and 1.568);
    \node at (3.200,2.800) {\tiny $3\pi/4$};
    \draw[very thick,blue!70!black,rotate around={-45:(3.200,0.800)}] (3.200,0.800) ellipse (0.433 and 1.045);
    \node at (1.200,0.800) {\tiny $3\pi/4$};
    \node at (2,-0.7) {\small $(3\pi/4,3\pi/4)$};
  \end{scope}
  \node at (5,2) {$\odot$};
  \node at (11,2) {$=$};
\end{tikzpicture}
\caption{Two elements of the two-parameter submonoid of Remark~\ref{rem:multi} with $m=2$,
$a_1=0.6$, $\rho_1=\rho_2=\Arch$, and their product. The blocks are $I_1\times J_1$ (top left) and
$I_2\times J_2$ (bottom right), with no mass outside the blocks; inside the blocks are Archimedean
copulas $C^{\Arch}_{\beta_k}$ (supported on ellipses). Below the squares are the parameters
$(\beta_1,\beta_2)$: the product is computed blockwise, and the angles add coordinatewise,
$(\pi/4,\pi/2)\odot(\pi/2,\pi/4)=(3\pi/4,3\pi/4)$.}
\label{fig:blocks}
\end{figure}

\begin{definition}[normalized inversion number]\label{def:inv}
Let $(X,Y),(X',Y')$ be independent points with distribution $P$. The pair forms an
\emph{inversion} if it is ordered the same way in both coordinates, and the \emph{normalized
inversion number} of the permuton $P$ is the probability of this event:
\[
  \inv(P):=\Prob{(X-X')(Y-Y')>0}.
\]
From now on we drop the word ``normalized''.
\end{definition}

The definition extends the discrete one. For $\sigma\in S_n$ the inversion number $\inv(\sigma)$
is the number of pairs $i<j$ with $\sigma(i)>\sigma(j)$; in the context of Coxeter groups, this
is the same as the permutation's length. Associate with $\sigma$ the permuton
$P_\sigma$ of \S\ref{sec:defs}: squares of side $1/n$ in place of the unit entries of the
permutation matrix, row $i$ counted from the top. Two points of $P_\sigma$ from different
squares are ordered the same way in both coordinates exactly when the corresponding pair of
indices is an inversion of $\sigma$, hence $\inv(P_\sigma)=2\inv(\sigma)/n^2+O(1/n)$: this is
the inversion number normalized by $\binom n2$, up to $O(1/n)$. The functional $\inv$ is
continuous in the weak topology on permutons (the boundary of the event $\{(X-X')(Y-Y')>0\}$
lies on the lines $X=X'$, $Y=Y'$, which have $P\otimes P$-measure zero), so $\inv$ of a limit
shape is the limit of the normalized inversion numbers. The limit shapes of the identity and
the reverse identity permutations are the permutons $W$ and $M$: $\inv(W)=0$, $\inv(M)=1$, as it should be;
$\inv(\Unif)=\tfrac12$. Under the identification of $\sigma$ with its graph
$\{(i/n,\sigma(i)/n)\}$, adopted in~\cite{AHRV2007} and in \S\ref{sec:rsn}, the sign is
reversed and $\inv$ counts the non-inverted pairs.

\begin{lemma}[geodesic parameter $=$ inversion number]\label{lem:kendall}
$\inv(C_\beta^\rho)=\beta/\pi$.
\end{lemma}

\begin{proof}
Let $G,G'\sim\rho$ be independent; a pair of points of the permuton is $(X_0,-X_\beta)$ and
$(X_0',-X_\beta')$; by Definition~\ref{def:inv}, an inversion means
$(X_0-X_0')\big((-X_\beta)-(-X_\beta')\big)>0$, that is, $(X_0-X_0')(X_\beta-X_\beta')<0$. The
difference $\Delta:=G-G'$ is rotation-invariant and absolutely continuous, so its direction
$\Theta$ is uniform and $\Delta\ne0$ almost surely. Since
$X_\theta-X_\theta'=|\Delta|\cos(\Theta-\theta)$, the sign of the product equals the sign of
$\cos\Theta\cos(\Theta-\beta)$, which is negative exactly for
$\Theta\in(\tfrac\pi2,\tfrac\pi2+\beta)\cup(\tfrac{3\pi}2,\tfrac{3\pi}2+\beta)$ --- a set of
total length $2\beta$. Hence $\inv(C_\beta^\rho)=\Prob{(X_0-X_0')(X_\beta-X_\beta')<0}=\beta/\pi$.
\end{proof}

Thus the geodesic parameter is the inversion number, and within the family $C^\rho$ it holds
that $\inv(C_{\beta_1}^\rho\odot C_{\beta_2}^\rho)=\min\big(\inv(C_{\beta_1}^\rho)+\inv(C_{\beta_2}^\rho),1\big)$.

\section{The Gaussian family and roots of independence}\label{sec:gauss}

By Example~\ref{ex:families}(a), the family $C^{\cN}$ consists of the standard Gaussian copulas,
$C^{\cN}_\beta=C_{-\cos\beta}$, and $C^{\cN}_{\pi/2}=\Unif$.

\begin{corollary}[sticky roots of $\Unif$]\label{cor:unifroots}
For every $n\ge1$, the permuton $\Unif$ has a unique $n$-th sticky root within the Gaussian
family, the Gaussian copula with correlation $-\cos(\pi/(2n))$:
\[
  \boxed{\ \Unif^{\boxdot 1/n}=C_{\pi/(2n)}^{\cN}=C_{-\cos(\pi/(2n))}\ };
\]
in particular $\sqrt{\Unif}:=\Unif^{\boxdot1/2}=C_{-1/\sqrt2}$. As $n$ grows, the root tends to
the identity element $W$.
\end{corollary}

\begin{proof}
Corollary~\ref{cor:roots} with $\beta=\pi/2$ and Example~\ref{ex:families}(a).
\end{proof}

\begin{remark}[a tropical twin of the Ornstein--Uhlenbeck semigroup]\label{rem:OU}
The Gaussian scale $\{C_r\}$ is also closed under the Markov $*$-product of copulas of
Darsow--Nguyen--Olsen~\cite{Darsow1992}, under which correlations multiply: $C_r*C_{r'}=C_{rr'}$.
The sticky product is its tropical twin: the law $r\mapsto rr'$ is replaced by angle addition
at $r=-\cos\beta$.
\end{remark}

\section{Rigidity: only the Gaussian family passes through $\Unif$}\label{sec:rigidity}

\begin{lemma}[Herschel--Maxwell]\label{lem:HM}
Let the distribution of a vector $Z=(Z_1,Z_2)$ be rotation-invariant with independent
coordinates $Z_1,Z_2$. Then $Z\sim\cN(0,\sigma^2I_2)$ for some $\sigma\ge0$; if the distribution
is absolutely continuous, then $\sigma>0$.
\end{lemma}

\begin{proof}
This is a classical characterization~\cite[Theorem~0.0.1]{Bryc1995}; we give a derivation without
any moment assumptions. The characteristic function $\phi(\xi)=\E e^{i\langle\xi,Z\rangle}$
is radial: $\phi(\xi)=\psi(|\xi|)$, $\psi$ is continuous, $\psi(0)=1$, and real-valued, since
rotation by $\pi$ gives $Z\eqd-Z$. Independence of the coordinates gives
$\psi(\sqrt{s^2+t^2})=\psi(|s|)\psi(|t|)$, that is, $f(x):=\psi(\sqrt x)$ satisfies
$f(a+b)=f(a)f(b)$ for $a,b\ge0$. From $f(x)=f(x/2)^2\ge0$ and $f(x_0)=0\Rightarrow f(x_0/2^k)=0\to
f(0)=0$ it follows that $f>0$, so $\log f$ is a continuous additive function, $f(x)=e^{-\lambda x}$,
with $\lambda\ge0$ because $|\phi|\le1$. Thus $\phi(\xi)=e^{-\lambda|\xi|^2}$ is the
characteristic function of $\cN(0,2\lambda I_2)$. Absolute continuity rules out
$\lambda=0$.
\end{proof}

\begin{theorem}[rigidity]\label{thm:rigidity}
For an absolutely continuous rotation-invariant $\rho$,
\[
  \Unif\in\{C_\beta^\rho:\beta\in[0,\pi]\}\quad\Longleftrightarrow\quad
  \rho=\cN(0,\sigma^2 I_2),\ \sigma>0,
\]
and in this case $\Unif=C_{\pi/2}^\rho$.
\end{theorem}

\begin{proof}
($\Leftarrow$) The copula does not change under the scaling $G\mapsto\sigma G$
(Remark~\ref{rem:inv}), and
$C_{\pi/2}^{\cN}=\Unif$ (Example~\ref{ex:families}). ($\Rightarrow$) Let
$C_{\beta_0}^\rho=\Unif$. At the square's center, $\beta_0/2\pi=\Unif(\tfrac12,\tfrac12)=\tfrac14$
(Lemma~\ref{lem:center}), that is, $\beta_0=\pi/2$, and the copula of the pair
$(X_0,-X_{\pi/2})=(G_1,-G_2)$ is $\Unif$. For continuous marginals this is equivalent to the
independence $G_1\perp G_2$, and Lemma~\ref{lem:HM} gives $\rho=\cN(0,\sigma^2I_2)$, $\sigma>0$.
\end{proof}

The sticky product plays no part in the proof: rigidity is a property of the projection family
itself. There are no moment assumptions. For instance, the rotation-invariant Cauchy distribution has no
variance and Remark~\ref{rem:corr} does not apply to it, but the theorem does, and $\Unif$ is
not in its family: for $\beta=\pi/2$ the copula equals $\Unif$ only if the coordinates are
independent, and for the Cauchy distribution they are dependent.

\section{Non-uniqueness: a singular root of independence}\label{sec:nonuniqueness}

The uniqueness of roots in Corollaries~\ref{cor:roots} and~\ref{cor:unifroots} is a within-family
statement, and rigidity forbids only \emph{projection} non-Gaussian families through $\Unif$.
Globally the root is not unique: we exhibit a singular square root of $\Unif$.

Let $\varphi(s):=|s-\tfrac12|$, $h(z):=\tfrac12\max(z,0)^2$, and
$\Gamma(t):=\int_0^t\varphi$, so that $\Gamma(t)=\tfrac t2-\tfrac{t^2}2$ for $t\le\tfrac12$,
$\Gamma(t)=\tfrac18+\tfrac{(t-1/2)^2}2$ for $t\ge\tfrac12$, $\Gamma(1)=\tfrac14$. Put
\begin{equation}\label{eq:PQ}
  P^\ast:=\ind_{\{0<|u-v|<1/2\}}\,\rd u\,\rd v+\varphi(s)\,\rd s\big|_{\{v=u\}},\qquad
  Q:=\ind_{\{0<|u+v-1|<1/2\}}\,\rd u\,\rd v+\varphi(s)\,\rd s\big|_{\{u+v=1\}},
\end{equation}
where the singular parts are the images of the measure $\varphi(s)\,\rd s$ on $[0,1]$ under
$s\mapsto(s,s)$ and $s\mapsto(s,1-s)$. The measure $Q$ is the image of $P^\ast$ under the
reflection $(x,y)\mapsto(x,1-y)$ (Fig.~\ref{fig:pstar}).

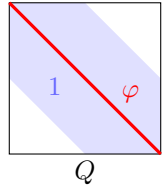
\begin{figure}[ht]\centering
\begin{tikzpicture}[scale=2.0,>=Latex]
  \fill[blue!12] (0,0.5)--(0.5,0)--(1,0)--(1,0.5)--(0.5,1)--(0,1)--cycle;
  \draw (0,0) rectangle (1,1);
  \draw[very thick,red] (0,1)--(1,0);
  \node[red] at (0.8,0.4) {\small $\varphi$};
  \node[blue!60] at (0.3,0.45) {\small $1$};
  \node at (0.5,-0.12) {\small $Q$};
\end{tikzpicture}
\caption{The singular root $Q$: density $1$ on the antidiagonal strip
$\{0<|u+v-1|<\tfrac12\}$ (total mass of the strip $\tfrac34$) and mass $\varphi(s)=|s-\tfrac12|$
on the antidiagonal (total mass $\tfrac14$).}
\label{fig:pstar}
\end{figure}

\begin{lemma}[copulas of $P^\ast$ and $Q$]\label{lem:Cpstar}
$P^\ast$ and $Q$ are permutons,
\begin{equation}\label{eq:Cpstar}
  C_{P^\ast}(u,v)=uv-h(u-\tfrac12)-h(v-\tfrac12)+h(u-\tfrac12-v)+h(v-\tfrac12-u)+\Gamma(\min(u,v)),
\end{equation}
$C_{P^\ast}$ is symmetric, $C_Q(u,v)=u-C_{P^\ast}(u,1-v)$, and $Q(\{u+v=1\})=\tfrac14$, that is,
$Q$ is singular.
\end{lemma}

\begin{proof}
The singular mass of $P^\ast$ in $[0,u]\times[0,v]$ equals $\Gamma(\min(u,v))$; the area of the
strip $\{|x-y|<\tfrac12\}$ inside the same rectangle equals $uv$ minus the areas of two corner
triangles, $h(u-\tfrac12)-h(u-\tfrac12-v)$ and $h(v-\tfrac12)-h(v-\tfrac12-u)$. This
gives~\eqref{eq:Cpstar}; substituting $v=1$ gives $C_{P^\ast}(u,1)=u$ in both cases
$u\lessgtr\tfrac12$, so $P^\ast$ is a permuton (symmetry is obvious). The reflection preserves
the Lebesgue marginals, so $Q$ is a permuton, and
$C_Q(u,v)=P^\ast([0,u]\times[1-v,1])=u-C_{P^\ast}(u,1-v)$. The mass of $Q$ on the antidiagonal
equals $\Gamma(1)=\tfrac14$.
\end{proof}

By formula~\eqref{eq:boxdot} and Lemma~\ref{lem:Cpstar},
\begin{equation}\label{eq:QQ}
  (C_Q\odot C_Q)(u,v)=\min_t\big[t-C_{P^\ast}(t,1-v)+u-C_{P^\ast}(u,t)\big]
  =u-\max_{t\in[0,1]}\Psi_{u,\,1-v}(t),
\end{equation}
where for $a,b\in[0,1]$ we put $\Psi_{a,b}(t):=C_{P^\ast}(t,a)+C_{P^\ast}(t,b)-t$ (the symmetry
of $C_{P^\ast}$ was used). Thus $Q\boxdot Q=\Unif$ is equivalent to the identity
$\max_t\Psi_{a,b}=ab$.

\begin{lemma}[concavity]\label{lem:concave}
For each $w\in[0,1]$ the function $t\mapsto C_{P^\ast}(t,w)$ is concave on $[0,1]$.
\end{lemma}

\begin{proof}
From~\eqref{eq:Cpstar}, for $t\ne w$
$\partial_tC_{P^\ast}(t,w)=L_w(t)+\varphi(t)\ind[t<w]$, where
$L_w(t)=\big(\min(w,t+\tfrac12)-\max(0,t-\tfrac12)\big)^+$ is the length of the cross-section of
the strip. For $t>w+\tfrac12$ the cross-section is empty and $\partial_tC_{P^\ast}=0$; for
$t<w+\tfrac12$, almost everywhere,
$\partial_t^2C_{P^\ast}=\ind[t<w-\tfrac12]-\ind[t>\tfrac12]+\varphi'(t)\ind[t<w]\le0$ (consider
$t\lessgtr\tfrac12$ with $\varphi'=\mp1$), and the only jump of the derivative is at $t=w$,
downward by $\varphi(w)\ge0$. The derivative is nonincreasing, so the function is concave.
\end{proof}

\begin{proposition}\label{prop:pstarroot}
$\max_{t\in[0,1]}\Psi_{a,b}(t)=ab$ for all $a,b\in[0,1]$.
\end{proposition}

\begin{proof}
By Lemma~\ref{lem:concave}, $\Psi_{a,b}$ is concave, so it suffices to exhibit a point $t^\ast$
with $\Psi_{a,b}(t^\ast)=ab$ and $0\in\partial\Psi_{a,b}(t^\ast)$. Let $a\le b$ and
$t^\ast:=\operatorname{median}\{a,\,b,\,a+b-\tfrac12\}\in[a,b]$. The one-sided derivatives are
$\Psi'_\pm(t^\ast)=L_a(t^\ast)+L_b(t^\ast)-1+\varphi(t^\ast)\big(\ind_\pm[t^\ast<a]+\ind_\pm[t^\ast<b]\big)$.
We consider three regimes according to the signs of $a-\tfrac12$, $b-\tfrac12$: for
$a\le b\le\tfrac12$ ($t^\ast=a$)
$\Psi(a)=(ab+\tfrac a2-\tfrac{a^2}2)+(\tfrac{a^2}2+\tfrac a2)-a=ab$, $\Psi'_+(a)=b-\tfrac12\le0$,
$\Psi'_-(a)=b-a\ge0$; for $a\le\tfrac12\le b$ ($t^\ast=a+b-\tfrac12$, a point of
differentiability) $\Psi'(t^\ast)=0$ and $\Psi(t^\ast)=ab$; the regime $a,b\ge\tfrac12$ is
symmetric to the first. The case analysis was carried out for pairs off the regime boundaries
($a,b\ne\tfrac12$, $a\ne b$, $a+b\ne1$); these form a dense set, and both sides of the identity
are continuous, hence it holds everywhere.
\end{proof}

\begin{theorem}[global non-uniqueness of $\sqrt{\Unif}$]\label{thm:nonuniqueness}
The permuton $\Unif$ has at least two distinct sticky square roots: the singular $Q$ and the
absolutely continuous Gaussian permuton with copula $C_{-1/\sqrt2}$.
\end{theorem}

\begin{proof}
$Q\boxdot Q=\Unif$ by~\eqref{eq:QQ} and Proposition~\ref{prop:pstarroot}: $u-u(1-v)=uv$. The
second root is Corollary~\ref{cor:unifroots}. The roots are distinct: $Q$ carries mass $\tfrac14$
on the antidiagonal, whereas the Gaussian permuton is absolutely continuous.
\end{proof}

\begin{remark}
The uniqueness of roots in Corollaries~\ref{cor:roots} and~\ref{cor:unifroots} is a statement
about roots within a geodesic; outside it the root is not unique already for $\Unif$.
\end{remark}

\section{The Archimedean family and random sorting networks}\label{sec:rsn}

Let $\rho=\Arch$ be the Archimedean measure of Example~\ref{ex:families}(b),
$G\sim\Arch$, and $X_\theta=\langle G,u_\theta\rangle$ as in \S\ref{sec:family}. For
$s\in[0,1]$ denote by $\Arch_s$ the joint distribution of the pair of projections
$(X_0,X_{\pi s})$ onto directions at angle $\pi s$; in the coordinates $G=(X,Y)$ this is the
pair $(X,\ X\cos\pi s+Y\sin\pi s)$.

\begin{theorem}[Angel--Holroyd--Romik--Vir\'ag, Dauvergne~\cite{AHRV2007,Dauvergne2022}]\label{thm:ahrv}
For each $s\in[0,1]$, the permutation of the uniformly random sorting network of size $n$ at
time $s$, regarded as a permuton, converges as $n\to\infty$, weakly in probability, to the
affine image of $\Arch_s$ in the unit square.
\end{theorem}

We denote this limit by $\mu_s$: it is the limit shape of the permutation of the network at time
$s$. The theorem is due to the cited authors; our contribution is Corollary~\ref{cor:rsn}.

\begin{lemma}\label{lem:mutau}
The copula of $\mu_s$ equals $C_{\pi(1-s)}^{\Arch}$; in particular $\mu_0=M$ and $\mu_1=W$.
\end{lemma}

\begin{proof}
By Remark~\ref{rem:inv}, the copula of $\mu_s$ is the copula of the pair $(X_0,X_{\pi s})$, the projections onto directions at angle
$\pi s$. By Definition~\ref{def:proj}, the copula $C_\beta^\rho$ corresponds to the angle
$\pi-\beta$; hence $\beta=\pi(1-s)$.
\end{proof}

\begin{corollary}[sticky products of the limit shapes]\label{cor:rsn}
For all $s_1,s_2\in[0,1]$
\[
  \boxed{\ \mu_{s_1}\boxdot\mu_{s_2}=\mu_{\max(s_1+s_2-1,\,0)}\ }.
\]
That is, the limit shapes of the permutations of the sorting network form a geodesic, traversed
in reverse time
from the identity element $\mu_1=W$ to the absorbing element $\mu_0=M$, with $\inv(\mu_s)=1-s$.
\end{corollary}

\begin{proof}
Theorem~\ref{thm:main} with $\rho=\Arch$ and Lemma~\ref{lem:mutau}:
$\min(\pi(1-s_1)+\pi(1-s_2),\pi)=\pi\big(1-\max(s_1+s_2-1,0)\big)$. By Lemma~\ref{lem:kendall},
$\inv(\mu_s)=\pi(1-s)/\pi=1-s$.
\end{proof}

The following construction is known. A network obtained by rotating a set of points in the
plane is called stretchable in~\cite{AHRV2007} and a \emph{geometric sorting network} by
Dauvergne~\cite{Dauvergne2022}. By Dauvergne's subnetwork theorem, a random $m$-out-of-$n$
subnetwork of the uniformly random sorting network converges in distribution, as $n\to\infty$,
to the geometric network of $m$ independent points with distribution $\Arch$. The
uniformly random network itself is, with high probability, not geometric~\cite{AGH2012}. We
consider the same construction for an arbitrary profile $\rho$: item (i) below is known, while
items (ii) and (iii) give the limit shapes of its permutations and single out the Gaussian case.

\begin{proposition}[geometric sorting network]\label{prop:geom}
Let $p_1,\dots,p_n$ be independent points with distribution $\rho$, and for each
$\theta\in[0,\pi]$ order the points by the value of the projection $\langle p_i,u_\theta\rangle$;
let $\sigma_n(\theta)\in S_n$ be the permutation taking the order at $\theta=0$ to the order at
$\theta$. Then almost surely:
\begin{enumerate}
  \item[\rm(i)] as $\theta$ increases from $0$ to $\pi$, the order changes by adjacent
    transpositions, each pair swaps exactly once, and $\sigma_n(\pi)$ is the reverse identity
    permutation; in other words, $\theta\mapsto\sigma_n(\theta)$ is a sorting network;
  \item[\rm(ii)] the permuton of the permutation $\sigma_n(\theta)$ built from its graph
    $\{(i/n,\sigma_n(\theta)(i)/n)\}$, as in Theorem~\ref{thm:ahrv}, converges, as $n\to\infty$,
    to the permuton with copula $C^\rho_{\pi-\theta}$, and the fraction of transpositions
    performed by angle $\theta$ converges to $\theta/\pi$;
  \item[\rm(iii)] for $\rho=\cN(0,I_2)$ the permutation $\sigma_n(\pi/2)$ is uniformly
    distributed on $S_n$ for every $n$; among absolutely continuous rotation-invariant $\rho$,
    the limit of $\sigma_n(\pi/2)$ equals $\Unif$ only for the Gaussian one.
\end{enumerate}
\end{proposition}

\begin{proof}
(i) A pair $i,j$ changes order exactly at the $\theta_{ij}\in[0,\pi)$ for which
$u_{\theta}\perp p_i-p_j$; the directions $p_i-p_j$ are almost surely pairwise distinct, and at
the moment $\theta_{ij}$ no third point has the same projection, so two adjacent points swap.
At $\theta=\pi$ all projections change sign.
(ii) $\sigma_n(\theta)$ is the rank permutation of $n$ independent pairs
$(\langle p_i,u_0\rangle,\langle p_i,u_\theta\rangle)$, that is, a $\mu$-random permutation
for the permuton $\mu$ whose copula is that of the pair of projections at angle $\theta$, namely
$C^\rho_{\pi-\theta}$ (Definition~\ref{def:proj}); convergence is Theorem~\ref{thm:sampling}
applied to the reflected permuton: the permuton built from the graph and the permuton $P_\sigma$
of \S\ref{sec:defs} differ by the reflection $v\mapsto1-v$.
The angle $\theta_{ij}$ is the direction of the vector $p_i-p_j$ modulo $\pi$, uniform on
$[0,\pi)$ by rotation invariance; the fraction of pairs with $\theta_{ij}\le\theta$ is a
$U$-statistic converging to $\theta/\pi$.
(iii) For Gaussian points the $x$ and $y$ coordinates are independent, so the rank permutations
by $x$ and by $y$ are independent and uniform, and so is their composition $\sigma_n(\pi/2)$. The
converse is Theorem~\ref{thm:rigidity} with $\beta=\pi/2$.
\end{proof}

That is, the limit shapes of the permutations of the geometric network on points with distribution $\rho$ form the
geodesic $C^\rho$, and Theorem~\ref{thm:main} describes their sticky products. For
$\rho=\Arch$ these limit shapes are exactly the $\mu_s$ of Lemma~\ref{lem:mutau}, in agreement
with Dauvergne's subnetwork theorem. For $\rho=\cN$ one
obtains a sorting network whose permutation after rotation by the angle $\pi/2$ is uniformly
distributed on $S_n$ (the number of transpositions performed by that angle is random and close
to $\tfrac12\binom n2$).

\begin{corollary}[the Gaussian and Archimedean geodesics are distinct]\label{cor:manygeo}
The Gaussian and Archimedean families $C^{\cN}$ and $C^{\Arch}$ are two distinct geodesics in the
projection class $\{C^\rho\}$: by Theorem~\ref{thm:rigidity}, only the Gaussian one passes through $\Unif$.
\end{corollary}

\begin{proposition}[a continuum of geodesics]\label{prop:continuum}
There is a continuum of distinct geodesics of the form $C^\rho$.
\end{proposition}

\begin{proof}
Two families $C^\rho$, $C^{\rho'}$ coincide as sets only if $C_\beta^\rho=C_\beta^{\rho'}$ for
all $\beta$, since $\beta$ is recovered from the value at the square's center (Lemma~\ref{lem:center}).
Let $\rho_r$, $r\in(0,1)$, be the uniform measure on the annulus $\{r\le|z|\le1\}$, and $F_r$
its projection marginal; $F_r$ is a homeomorphism of $[-1,1]$ onto $[0,1]$. The permuton with
copula $C_{\pi/2}^{\rho_r}$ is the image of $\rho_r$ under the homeomorphism
$(x,y)\mapsto(F_r(x),1-F_r(y))$ of the square $[-1,1]^2$ onto $[0,1]^2$, so its support is the
image of the annulus, and the section of the support by the line $v=\tfrac12$ (the image of the
axis $y=0$) is the complement of an interval of length $\ell(r)=\rho_r\{|x|<r\}$. The function
$\ell$ is continuous on $(0,1)$, positive, and tends to zero as $r\to0$, hence takes a continuum
of values; distinct values of $\ell$ correspond to distinct supports, that is, to distinct
copulas $C_{\pi/2}^{\rho_r}$.
\end{proof}

\begin{remark}[Demazure product and the oriented swap process]\label{rem:osp}
The oriented swap process of Angel, Holroyd, and Romik~\cite{AHR2009} is another random
dynamics on $S_n$:
every adjacent pair standing in increasing order swaps at unit rate until the reverse identity
permutation is reached. Defant~\cite{Defant2025} showed that the permutation of this process at
time $\alpha n$ has a limiting permuton $\lambda^\alpha$, and observed that the Demazure product
acts on limit shapes: if permutations $v_n$ converge to a permuton $\mu$, then the Demazure
products of $v_n$ with the permutation of the process converge to $\mu\boxdot\lambda^\alpha$;
applying this to an independent copy of the process, one gets
$\lambda^{\alpha'}\boxdot\lambda^{\alpha}=\lambda^{\alpha+\alpha'}$, that is, the family
$\{\lambda^\alpha\}$ is a semigroup under the sticky product. Corollary~\ref{cor:rsn} is a
statement of the same type for a different dynamics, the uniformly random sorting network, and
a different semigroup: the permuton $\lambda^\alpha$ carries a singular part and differs from
$\mu_s$; whether it is the projection copula of some profile is unknown.
\end{remark}

\section{Context and open questions}\label{sec:context}

\paragraph{The arctic ellipse.} If $\rho$ is supported in the unit disk, then the pair of
projections $(X_0,X_{\pi-\beta})$ is supported in the ellipse
$\{x^2-2xy\cos(\pi-\beta)+y^2\le\sin^2(\pi-\beta)\}$, the image of the disk under the linear map
$G\mapsto(X_0,X_{\pi-\beta})$; these are the ellipses of the bottom row of
Fig.~\ref{fig:families}. The same curve is the arctic ellipse of the domain-wall six-vertex
model~\cite{Colomo2010}: the boundary between the frozen and disordered regions. The
coincidence of the shape of the support is an observation, not a theorem; a connection between
the two models is unknown to us.

\paragraph{Lax--Oleinik semigroups.} The sticky product is the $(\min,+)$ convolution of
distribution functions. The same convolutions, only in the $(\max,+)$ semiring, underlie the theory of
surface growth:
the metric composition of the directed landscape~\cite{DOV2022},
\[
  L(x,r;y,t)=\max_z\bigl[L(x,r;z,\tau)+L(z,\tau;y,t)\bigr],
\]
and the Lax--Oleinik semigroup of the KPZ
fixed point~\cite{MQR2021}. There a one-parameter semigroup is generated by a single random
function; here the geodesic $C^\rho$ is generated by a single profile $\rho$.

\paragraph{The $q\to0$ limit.} At the level of integrable models, the Demazure product, and
with it $\boxdot$, arises as the $q\to0$ limit of the colored stochastic six-vertex model, in
which permutations are built from $q$-deformed Demazure products~\cite{Tikhonov2026}. A natural
question is whether a $q$-deformation of the geodesics $C^\rho$ exists.

\paragraph{Open questions.}
\begin{itemize}
  \item Does there exist a one-parameter $\boxdot$-submonoid from $W$ to $M$ passing through
    $\Unif$ and different from the Gaussian one (for instance, containing the singular root
    $Q$)? The projection construction does not provide one (Theorem~\ref{thm:rigidity}). A
    natural candidate is the family of Mallows permutons~\cite{Starr2009}: the continuous-time
    Mallows process~\cite{Corsini2022,AdamczakKotowski2024} goes from the identity permutation
    to the reverse identity permutation, is exactly uniform at time $t=1$, and its limit shapes
    form a path from $M$ to $W$ through $\Unif$. Whether this path is a $\boxdot$-geodesic is
    unknown.
  \item By Rahman--Vir\'ag--Vizer~\cite{RVV}, the Archimedean path is the unique
    minimal-energy path between the diagonal and the antidiagonal in the space of permutons with
    the Wasserstein distance; here it also turns out to be a
    $\boxdot$-geodesic. Are these structures compatible for other profiles?
  \item \emph{An intrinsic Gaussian sorting network.} The geometric network of
    Proposition~\ref{prop:geom} defines a distribution on sorting networks through external
    random points. For the Archimedean measure the same limit is attained by an intrinsic
    measure --- the uniform measure on all sorting networks. Is there an intrinsic random
    sorting network (uniform on a natural class, or Markov in time) with a Gaussian limit, that
    is, uniformly random in the middle? The Mallows process is uniform at time $t=1$ but uses
    all transpositions, not only adjacent ones~\cite{Corsini2022}.
  \item Is the limiting permuton $\lambda^\alpha$ of the oriented swap process a projection
    copula of some, not necessarily absolutely continuous, rotation-invariant profile
    (Remark~\ref{rem:hyp})?
\end{itemize}


\begin{thebibliography}{99}

\bibitem{AdamczakKotowski2024} R.~Adamczak, M.~Kotowski.
  \emph{The global and local limit of the continuous-time Mallows process.} arXiv:2404.08554 (2024).

\bibitem{AGH2012} O.~Angel, V.~Gorin, A.~E.~Holroyd.
  \emph{A pattern theorem for random sorting networks.}
  Electron.\ J.\ Probab.\ \textbf{17} (2012), no.~99, 1--16.

\bibitem{AHR2009} O.~Angel, A.~E.~Holroyd, D.~Romik.
  \emph{The oriented swap process.} Ann.\ Probab.\ \textbf{37}(5) (2009), 1970--1998.

\bibitem{AHRV2007} O.~Angel, A.~E.~Holroyd, D.~Romik, B.~Vir\'ag.
  \emph{Random sorting networks.} Adv.\ Math.\ \textbf{215}(2) (2007), 839--868.

\bibitem{BjornerBrenti2005} A.~Bj\"orner, F.~Brenti. \emph{Combinatorics of Coxeter Groups.}
  Graduate Texts in Mathematics \textbf{231}, Springer, New York, 2005.

\bibitem{Bryc1995} W.~Bryc. \emph{The Normal Distribution: Characterizations with Applications.}
  Lecture Notes in Statistics \textbf{100}, Springer, 1995.

\bibitem{Colomo2010} F.~Colomo, A.~G.~Pronko.
  \emph{The arctic curve of the domain-wall six-vertex model.}
  J.~Stat.\ Phys.\ \textbf{138} (2010), 662--700.

\bibitem{Corsini2022} B.~Corsini. \emph{Continuous-time Mallows processes.} arXiv:2205.04967 (2022).

\bibitem{Darsow1992} W.~F.~Darsow, B.~Nguyen, E.~T.~Olsen.
  \emph{Copulas and Markov processes.} Illinois J.~Math.\ \textbf{36}(4) (1992), 600--642.

\bibitem{Dauvergne2022} D.~Dauvergne.
  \emph{The Archimedean limit of random sorting networks.}
  J.~Amer.\ Math.\ Soc.\ \textbf{35} (2022), 1215--1267.

\bibitem{DOV2022} D.~Dauvergne, J.~Ortmann, B.~Vir\'ag.
  \emph{The directed landscape.} Acta Math.\ \textbf{229}(2) (2022), 201--285.

\bibitem{Defant2025} C.~Defant. \emph{Permutons from Demazure products.} arXiv:2505.15630 (2025).

\bibitem{Ehresmann1934} C.~Ehresmann. \emph{Sur la topologie de certains espaces homog\`enes.}
  Ann.\ of Math.\ (2) \textbf{35}(2) (1934), 396--443.

\bibitem{HKMRS2013} C.~Hoppen, Y.~Kohayakawa, C.~G.~Moreira, B.~R\'ath, R.~M.~Sampaio.
  \emph{Limits of permutation sequences.} J.~Combin.\ Theory Ser.~A \textbf{120}(1) (2013), 93--113.

\bibitem{MQR2021} K.~Matetski, J.~Quastel, D.~Remenik.
  \emph{The KPZ fixed point.} Acta Math.\ \textbf{227}(1) (2021), 115--203.

\bibitem{Pflueger2022} N.~Pflueger. \emph{An extended Demazure product on integer permutations via
  min-plus matrix multiplication.} arXiv:2206.14227 (2022).

\bibitem{PresnyakovThesis} S.~Presnyakov. \emph{Continualization of the Hecke monoid} (in Russian).
  Bachelor's thesis, St.~Petersburg State University, 2025. Advisor: A.~Tiskin.

\bibitem{PresnyakovTiskin} S.~Presnyakov, A.~Tiskin.
  \emph{The Permuton Hecke Monoid: An Algebraic Framework for Weighted Sequence Alignment.}
  In: International Conference Polynomial Computer Algebra~'2025, St.~Petersburg,
  September~29 -- October~4, 2025. Euler International Mathematical Institute, 2025,
  pp.~97--99. \url{https://pca.conf-pdmi.ru/2025/}

\bibitem{Nelsen2006} R.~B.~Nelsen. \emph{An Introduction to Copulas.} 2nd ed., Springer, 2006.

\bibitem{RVV} M.~Rahman, B.~Vir\'ag, M.~Vizer.
  \emph{Geometry of permutation limits.} Combinatorica \textbf{39}(4) (2019), 933--960.

\bibitem{Starr2009} S.~Starr. \emph{Thermodynamic limit for the Mallows model on $S_n$.}
  J.~Math.\ Phys.\ \textbf{50}(9) (2009), 095208.

\bibitem{tiskin2008semi} A.~Tiskin. \emph{Semi-local string comparison: algorithmic techniques
  and applications.} Math.\ Comput.\ Sci.\ \textbf{1} (2008), 571--603.

\bibitem{tiskin2015fast} A.~Tiskin. \emph{Fast distance multiplication of unit-Monge matrices.}
  Algorithmica \textbf{71}(4) (2015), 859--888.

\bibitem{Tikhonov2026} M.~Tikhonov. \emph{Random permutations from $q$-Demazure products.}
  arXiv:2604.06532 (2026).

\end{thebibliography}
\end{document}